\documentclass[12pt]{article}

\usepackage{amsthm,amsmath}

\newtheorem{lemma}{Lemma}
\newtheorem{theorem}{Theorem}
\newtheorem{corollary}{Corollary}
\newtheorem{conjecture}{Conjecture}

\title{Classifying Rotationally-Closed Languages Having Greedy Universal Cycles}
\author{Joseph DiMuro}
\date{May 29, 2018}

\begin{document}
\maketitle

\begin{abstract}
  Let $\textbf{T}(n,k)$ be the set of strings of length $n$ over the alphabet $\Sigma=\{1,2,\ldots,k\}$. A universal cycle for $\textbf{T}(n,k)$ can be constructed using a greedy algorithm: start with the string $k^n$, and continually append the least symbol possible without repeating a substring of length $n$. This construction also creates universal cycles for some subsets $\textbf{S}\subseteq\textbf{T}(n,k)$; we will classify all such subsets that are closed under rotations.
\end{abstract}

Let $\textbf{T}(n,k)$ be the set of strings of length $n$ over the alphabet $\Sigma=\{1,2,\ldots,k\}$. Given a subset $\textbf{S}\subseteq\textbf{T}(n,k)$, we will be interested in finding a ``universal cycle'' for $\textbf{S}$: that is, a string of length $|\textbf{S}|$ where each element of $\textbf{S}$ occurs exactly once in that string when it is viewed as a cycle.
  
For example, we could consider the subset $\textbf{S}_1\subseteq\textbf{T}(3,5)$ consisting of ascending strings and their rotations. That is,
$$\textbf{S}_1=\{123,124,125,134,135,145,231,234,235,241,245,251,312,341,342,$$
$$345,351,352,412,413,423,451,452,453,512,513,514,523,524,534\}$$
  
Below is a universal cycle for $\textbf{S}_1$. To make it easier to locate certain strings in the cycle, the string 534 is repeated at both the start and the end of the cycle. (The 534 is put in parentheses at the start as a reminder that the 534 is repeated.)
$$(534)123124134234512513514523524534$$

This universal cycle was constructed using a greedy algorithm: after choosing 534 as the starting string, each subsequent digit was chosen by looking for the smallest digit that could be chosen without duplicating any length-3 substrings (except for 534 itself, to conclude the cycle).
  
As another example, consider the subset $\textbf{S}_2\subseteq\textbf{T}(2,9)$ consisting of all two-digit strings where both the string and its reverse yield composite numbers in base 10. For example, $43\not\in\textbf{S}_2$ because 43 is prime; $34\not\in\textbf{S}_2$ because 34 is the reverse of a prime number.
$$\textbf{S}_2=\{12,15,18,21,22,24,25,26,27,28,33,36,39,42,44,45,$$
$$46,48,49,51,52,54,55,56,57,58,62,63,64,65,66,68,$$
$$69,72,75,77,78,81,82,84,85,86,87,88,93,94,96,99\}$$

If we try to construct a universal cycle for $\textbf{S}_2$ using a greedy algorithm starting from 99, here's what we get:
$$(99)33621224251526394454648182728496556685758699$$
Unfortunately, this cycle isn't quite universal: 77, 78, 87, and 88 are missing. This raises the question: can we find necessary and sufficient conditions on $\textbf{S}\subseteq\textbf{T}(n,k)$ so that a universal cycle for $\textbf{S}$ can be generated from a greedy algorithm? In this paper, we will find such conditions, under the assumption that $\textbf{S}$ is closed under rotations.
  
\section{Results}

Some notational conventions: when we are working with a particular set $\textbf{T}(n,k)$, we will use $\alpha$ and $\beta$ (possibly with subscripts) to represent strings in $\textbf{T}(n,k)$. Other Greek letters (like $\gamma$) will represent strings in $\textbf{T}(m,k)$ for some $m\le n$. (This includes the possibility of $\gamma$ being an empty string.) Latin letters ($a$, $b$, etc.) will represent individual elements of $\{1,\ldots,k\}$. We will sometimes use exponential notation to write strings with repetitions in a shorter form. A couple of examples: $2^5 3$ represents the string $222223$, and $4(21)^3$ represents the string $4212121$.

Given a set $\textbf{S}\subseteq\textbf{T}(n,k)$ (where $\textbf{S}\ne\emptyset$), let $\alpha\in\textbf{S}$. Define the string $\textit{Greedy}_\alpha(\textbf{S})$ as follows: let $\beta_0=\alpha$, and having defined $\beta_0$ through $\beta_j=a\gamma$, let $\beta_{j+1}=\gamma a_{j+1}$, where $a_{j+1}$ is the least element of $\{1,\ldots,k\}$ such that $\gamma a_{j+1}\in\textbf{S}$ and $\gamma a_{j+1}\ne\beta_i$ for all $1\le i\le j$. The process halts when we reach an $\beta_m$ where either $\beta_m=\alpha$, or it is impossible to define $\beta_{m+1}$. (The latter occurs when, given $\beta_m=a\gamma$, we have $\gamma b\in\{\beta_i\}_{i=1}^m$ for all $b\in\{1,\ldots,k\}$ such that $\gamma b\in\textbf{S}$.) Let $m$ be the largest positive integer for which $\beta_m$ exists. We then define $\textit{Greedy}_\alpha(\textbf{S})$ to be: $$\textit{Greedy}_\alpha(\textbf{S})=a_1 a_2\cdots a_m.$$ In the case where the process ends because $\beta_m=\alpha$, the string $\textit{Greedy}_\alpha(\textbf{S})$ may be viewed as a cycle; each of the strings from $\beta_1$ to $\beta_m=\alpha$ appears exactly once in that cycle.

Let $\mathcal{G}(n,k)$ be the collection of subsets $\textbf{S}\subseteq\textbf{T}(n,k)$ such that, for some $\alpha\in S$, $\textit{Greedy}_\alpha(\textbf{S})$ is a universal cycle for $\textbf{S}$. That is, the length-$n$ suffix of $\textit{Greedy}_\alpha(\textbf{S})$ is $\alpha$, and for all $\beta\in\textbf{S}$, $\beta$ is a substring of $\textit{Greedy}_\alpha(\textbf{S})$ (treated as a cycle). The goal is to find a characterization of the sets in $\mathcal{G}(n,k)$ that are closed under rotations.

Given $\textbf{S}\in\textbf{T}(n,k)$, and given $\alpha,\beta\in\textbf{S}$, we will say that $\beta$ is ``increasable in $\textbf{S}$ to $\alpha$'' if we can transform $\beta$ into $\alpha$ by continually increasing individual symbols, and if the resulting string after each such increase is in $\textbf{S}$. (By convention, we will say that $\alpha\in\textbf{S}$ is increasable in $\textbf{S}$ to $\alpha$.) For example, for our ``no primes'' set $\textbf{S}_2\subseteq\textbf{T}(2,9)$, 57 is increasable in $\textbf{S}_2$ to 99:

$$57\rightarrow 58\rightarrow 68\rightarrow 69\rightarrow 99$$

Given this definition, our ultimate result is the following:

\begin{theorem} \label{mainresult}
  Let $\textbf{S}\subseteq\textbf{T}(n,k)$ be closed under rotations, and let $\alpha,\beta\in\textbf{S}$. Then $\beta$ is a substring of $\textit{Greedy}_\alpha(\textbf{S})$ (treated as a cycle) if and only if $\beta$ is increasable in $\textbf{S}$ to a rotation of $\alpha$.
  
  Thus, if $\textbf{S}\subseteq\textbf{T}(n,k)$ is closed under rotations, then $\textbf{S}\in\mathcal{G}(n,k)$ if and only if there exists an $\alpha\in\textbf{S}$ such that every $\beta\in\textbf{S}$ is increasable in $\textbf{S}$ to $\alpha$. 
\end{theorem}

This theorem explains the absence of 77, 78, 87, and 88 in the string $\textit{Greedy}_{99}(\textbf{S}_2)$: none of those four strings are increasable in $\textbf{S}_2$ to 99, since none of 79, 89, 97, or 98 is in $\textbf{S}$. Note that this theorem is a generalization of Theorem 3, from \cite{SWW}.

The proof of this result will rely on an analysis of a combinatorial game, which we'll call the Warden's Game. The rules for this game will be given in Section 2, and the proof of Theorem \ref{mainresult} will follow in Section 3. In Section 4, we will look at several interesting families of sets $\textbf{S}\subseteq\textbf{T}(n,k)$ where a universal cycle can be generated with the greedy algorithm. Lastly, a possible avenue for future work will be detailed in Section 5.

\section{The Warden's Game}

Consider the following fanciful scenario: there's a certain prison warden who loves playing games. He sometimes makes an offer to let his prisoners out of prison, if they can beat him at a particular game. The game works as follows:

The warden shows the prisoner a row of $n$ $k$-sided dice on a table. On each die, the faces are numbered from 1 to $k$. (It's possible to have $k=2$ here; the ``dice'' would then be coins with a 1 on one side and a 2 on the other.) A certain string $\alpha\in\textbf{T}(n,k)$ is chosen: the prisoner will earn his freedom if, after any move, the dice on the table are showing the string $\alpha$. (If the dice are showing $\alpha$ at the start of the game, the prisoner does not immediately win; the prisoner only wins when the dice show $\alpha$ after a move.) This game will be played at a rate of one move per day. So, the prisoner wants to reach $\alpha$ as quickly as possible; the warden wants to delay this as long as possible (indefinitely, if he can).

Each day, the rightmost die in the row will be moved to the far left, and possibly rotated to show a different number. The warden always has priority; he may transfer the rightmost die to the far left, and lower the number on that die. If he doesn't want to do that (or he can't, because he can't lower the number any further), then the warden passes; then the prisoner must transfer the rightmost die to the far left, and optionally increase the number on that die.

As an example: let $n=3$ and $k=6$, so that the game is being played with three 6-sided dice. Let's say the current position is 513; the leftmost die shows 5, the middle die shows 1, and the rightmost die shows 3. The warden may transfer the rightmost die to the far left, lowering its value to 1 or 2 (thus producing the position 151 or 251). Or the warden may pass, in which case the prisoner must transfer the rightmost die and optionally increase its value (producing one of the positions 351, 451, 551, or 651). Let's say the warden chooses to move to the position 251. Then on the next move, the warden can't lower the value showing on the rightmost die. So the warden must pass, and the prisoner can move to 125, 225, 325, 425, 525, or 625. And so on.

Note: in the case where $k=2$, the rules can be stated even more simply. If the rightmost coin is showing a 2, the warden transfers that coin, and optionally flips it to 1. If the rightmost coin is showing a 1, the prisoner transfers that coin, and optionally flips it to 2.

We can generalize this game still further, by limiting the legal positions in the game. We can choose any subset $\textbf{S}\subseteq\textbf{T}(n,k)$, closed under rotations, to be the set of legal positions. (We'll assume that the goal state $\alpha$ is in $\textbf{S}$.) Then each move of the game, whether made by the prisoner or the warden, must be to a position in $\textbf{S}$. We require that $\textbf{S}$ be closed under rotations so that there is a legal move from every legal position; if the warden ever passes, the prisoner always has the option to transfer the rightmost die without changing its value.

In \cite{GW}, Weiss analyzed the Warden's Game (though not under that name) in the case where $k=2$, $\textbf{S}=\textbf{T}(n,2)$, and $\alpha$ is the string $2^n$. Weiss proved that the game tree for the game is summarized by the lexicographically minimal de Bruijn sequence for $\textbf{T}(n,2)$; if both players play optimally, the game will proceed backwards through the de Bruijn sequence, one move at a time. For example, if $n=4$, the lexicographically minimal de Bruijn sequence is the following:

$$(2222)1111211221212222$$

For this game, consider the position 2212. If we move one step backwards in the de Bruijn sequence from 2212, we get 1221; thus, the optimal move from 2212 must be for the warden to flip the rightmost coin before moving it, producing the position 1221. Similarly, the next optimal move is for the prisoner to move from 1221 to 1122, and so on, until the goal position 2222 is finally reached.

As we will prove in Section 3, the same holds true for any values of $n$ and $k$, any subset $\textbf{S}\subseteq\textbf{T}(n,k)$ of legal positions (closed under rotations), and any goal state $\alpha\in\textbf{S}$. The greedy algorithm always generates the full game tree for the Warden's Game. As an example, let's once again consider the subset $\textbf{S}_2\subseteq\textbf{T}(2,9)$ consisting of those strings where both the string itself and its reverse are 2-digit composite numbers. Here, once again, is the (not quite universal) cycle generated by the greedy algorithm, starting from $\alpha=99$.

$$(99)33621224251526394454648182728496556685758699$$

For example, consider the position 82. The preceding substring of length 2 is 18; thus, the optimal move must be for the warden to move the rightmost die and reduce its value from 2 to 1. The next optimal move must be to 81; both the prisoner and the warden refuse to change the value on the rightmost die. The next optimal move is to 48, which means the warden passes, and the prisoner increases the value on the rightmost die from 1 to 4. And so on.

Remember, four positions from $\textbf{S}_2$ do not appear in this cycle: 77, 78, 87, and 88. Why don't they appear? Because they are losing positions for the prisoner! From any of those positions, the warden has a simple way to keep the game going indefinitely: he refuses to ever decrease the value on a die, and passes every time. The prisoner will never be able to increase a die to a 9, because 79, 89, 97, and 98 are all illegal positions. So the prisoner will never be able to reach the goal state, 99.

Our goal for the next section is to prove that this sort of thing happens regardless of the choices of $\textbf{S}$ and $\alpha$. We will show that the prisoner can win from a given position $\beta\in\textbf{S}$ if and only if he can win from $\beta$ with the warden always passing: this happens when $\beta$ is increasable in $\textbf{S}$ to a rotation of $\alpha$. We will also show that the greedy algorithm generates the game tree for this game; thus, the greedy algorithm generates a universal cycle if and only if every $\beta\in\textbf{S}$ is increasable in $\textbf{S}$ to a rotation of $\alpha$.

\section{Proof of Theorem 1}

Assume we are given values of $n$ and $k$, a set $\textbf{S}\subseteq\textbf{T}(n,k)$ of legal positions (closed under rotations), and a goal state $\alpha\in\textbf{S}$. Define the ``remoteness function'' $r$ on $\textbf{S}$ as follows: given $\beta\in\textbf{S}$, the remoteness of $\beta$, $r(\beta)$, is the number of moves the game will last starting from $\beta$ if both players play optimally. If the warden can keep the game going forever, then $r(\beta)=\infty$. This definition of remoteness is similar to the concept of remoteness used in \cite{BCG}.

Note: we can consider $\alpha$ to either be an end position (of remoteness 0) or a start position (of nonzero remoteness). We will always use the notation $r(\alpha)$ for the number of moves the game will last \textbf{starting} from $\alpha$; thus, $r(\alpha)>0$.

\begin{lemma} \label{ifwardenpasses}
  Let $\beta_1,\beta_2\in\textbf{S}$ be such that $\beta_1=\gamma a$ and $\beta_2=a'\gamma$ for $a'\ge a$. (Thus, if the current position is $\beta_1$ and the warden passes, then the prisoner may move to $\beta_2$.) Then, starting from $\beta_1$, the prisoner has a strategy which can force the position to eventually reach $\beta_2$. 
\end{lemma}

\begin{proof}
  This can be proven by induction on the sum of the symbols in $\beta_1$. Starting from $\beta_1=b_1 b_2\cdots b_{n-1}a$, if the warden passes, then the prisoner may move immediately to $\beta_2$. Otherwise, the warden must move to $\beta_3=a''b_1 b_2\cdots b_{n-1}$ for some $a''<a$. But the sum of the symbols of $\beta_3$ is less than the sum of the symbols of $\beta_1$. So by the inductive hypothesis, the prisoner has a strategy to eventually force the position to $b_{n-1}a''b_1\cdots b_{n-2}$, then to $b_{n-2}b_{n-1}a''b_1\cdots b_{n-3}$, and so on to $b_1 b_2\cdots b_{n-1}a''$. We still have a smaller sum than the sum of the symbols in $\beta_1$, so the prisoner can eventually force the position to $\beta_2=a'b_1 b_2\cdots b_{n-1}$, since $a'\ge a''$.
\end{proof}

\begin{lemma}
  Given $\beta\in\textbf{S}$, the prisoner can win from $\beta$ if and only if $\beta$ is increasable in $\textbf{S}$ to a rotation of $\alpha$.
\end{lemma}

\begin{proof}
  If $\beta$ is not increasable in $\textbf{S}$ to a rotation of $\alpha$, then the warden can keep the game going indefinitely, simply by passing on every turn. Since the prisoner can only increase values, if the warden always passes, the prisoner will only be able to reach positions $\gamma$ where $\beta$ is increasable in $\textbf{S}$ to a rotation of $\gamma$. Since $\alpha$ is not such a position, the prisoner can never win.
  
  Now assume that $\beta$ is increasable in $\textbf{S}$ to a rotation of $\alpha$. Then, if the warden chooses to pass on every move, then there is a sequence of moves $\beta=\beta_0,\beta_1,\beta_2,\ldots,\beta_m=\alpha$ that the prisoner may make to win. By Lemma \ref{ifwardenpasses}, if the game starts from $\beta=\beta_0$, then prisoner can eventually force the position to be $\beta_1$, then $\beta_2$, and so on until finally reaching $\alpha$ and winning.
\end{proof}

Note: while this shows that the prisoner can win eventually from any position $\beta$ that is increasable in $S$ to $\alpha$, the recursive strategy described above will probably \textbf{not} be the prisoner's optimal strategy.

\begin{lemma} \label{nojumps}
  Given any positive integer $m$, if there are no positions of remoteness $m$, then there are no positions of remoteness $m+1$. (Thus, by induction, there are no positions of remoteness $m'$ for any integer $m'\ge m$.)
\end{lemma}

\begin{proof}
  If there were a position of remoteness $m+1$, then with optimal play, the first move from such a position would be to a position of remoteness $m$... and no such position exists. So there are no positions of remoteness $m+1$.
\end{proof}

\begin{lemma} \label{largerhelpsdwarden}
  Given positions $\gamma a_1,\gamma a_2\in\textbf{S}$, if $a_1<a_2$, then $r(\gamma a_1)\le r(\gamma a_2)$.
\end{lemma}

The point here is that, the greater the rightmost symbol in the string, the better off the warden is. From $\gamma a_1$, the warden may move to any $a\gamma\in\textbf{S}$ such that $a<a_1$, or the warden may give the prisoner the choice to move to any $a\gamma\in\textbf{S}$ where $a\ge a_1$. From $\gamma a_2$, the warden still may move to any $a\gamma\in\textbf{S}$ such that $a<a_1$, or the warden can ensure that the next move is to $a\gamma\in\textbf{S}$ for some $a\ge a_1$... but in the latter case, the warden may choose a specific $a\gamma\in\textbf{S}$ such that $a_1\le a<a_2$, if he so desires. This extra option can only help the warden, never hurt him. So we must have $r(\gamma a_1)\le r(\gamma a_2)$.

Note: we will later see that if $r(\gamma a_1)$ and $r(\gamma a_2)$ are both finite, then $r(\gamma a_1)<r(\gamma a_2)$.

\begin{lemma} \label{onechain}
  For any nonnegative integer $m$, there is at most one position of remoteness $m$.
\end{lemma}

This is a significant result; combined with Lemma \ref{nojumps}, the conclusion is that the ``game tree'' is really a chain, not a tree. There is one position of remoteness 0 (namely, $\alpha$), one position of remoteness 1, one position of remoteness 2, and so on until all the winning positions for the prisoner have been exhausted. And given any position $\beta$ that is winning for the prisoner, if a game starting from $\beta$ is played optimally, the game will pass through all positions of remoteness less than $r(\beta)$ until finally reaching $\alpha$.

\begin{proof}
  We will prove this by contradiction. Let $m$ be the smallest integer where there are multiple positions of remoteness $m$. There is only one position of remoteness 0 (namely, $\alpha)$, so $m\ge 1$. Let $a\gamma$ be the one position of remoteness $m-1$; this position must be reachable in one move from all positions of remoteness $m$, so all such positions must have the form $\gamma b$.
  
  Let $b_1<b_2<\cdots <b_l$ be the elements of $\{1,2,\cdots,k\}$ such that $\gamma b_i\in\textbf{S}$ for each $i$. By Lemma \ref{largerhelpsdwarden}, $r(\gamma b_1)\le r(\gamma b_2)\le\cdots\le r(\gamma b_l)$. If $j$ is the smallest natural number such that $r(\gamma b_j)=m$, then because there is just one position of each remoteness less than $m$, we must have $$r(\gamma b_1)<r(\gamma b_2)<\cdots <r(\gamma b_{j-1})<r(\gamma b_j)=r(\gamma b_{j+1}).$$
  
  For each $i\le j$, let $a_i\gamma$ be the next position reached from $\gamma b_i$ if both sides play optimally (the $a_i$'s for $1\le i\le j$ are all distinct). We can now show that it is impossible to have $r(\gamma b_j)=r(\gamma b_{j+1})=m$:
  
  Consider the two sets $\{a_i\}_{i=1}^j$ and $\{b_i\}_{i=1}^j$. If these two sets are identical, then consider what happens if the warden passes from the position $\gamma b_{j+1}$. The prisoner is then forced to move to $a\gamma$ for some $a\ge b_{j+1}$. This $a$ will not be an element of $\{a_i\}_{i=1}^j$, and hence $a\gamma$ will not have remoteness at most $m-1$. So by passing, the warden can force the next move to be to a position of remoteness at least $m$; $\gamma b_{j+1}$ can't have remoteness $m$.
   
   On the other hand, assume $\{a_i\}_{i=1}^j$ and $\{b_i\}_{i=1}^j$ are not identical. That means there is some $b_i<b_{j+1}$ that is not in $\{a_i\}_{i=1}^j$. If the warden moves from $\gamma b_{j+1}$ to $b_i\gamma$, then the warden has not moved to a position of remoteness at most $m-1$. So again, the warden was able to force the next move to be to a position of remoteness at least $m$; $\gamma b_{j+1}$ can't have remoteness $m$.
   
   This completes the contradiction; it is impossible to have two positions of the same finite remoteness.
\end{proof}

\begin{lemma}
  If $\alpha\in\textbf{S}$ is the goal state, then the position in $\textbf{S}$ of highest finite remoteness is $\alpha$.
\end{lemma}

The reason: since $\alpha$ is (trivially) increasable in $\textbf{S}$ to a rotation of $\alpha$, $r(\alpha)$ is finite. Say $r(\alpha)=m>0$ (we are treating $\alpha$ as a start position, not an end position). If there were any position $\beta$ such that $r(\beta)=m+1$, then with optimal play, the next move from $\beta$ would be to a position of remoteness $m$: namely, $\alpha$. But that means, with optimal play, $\beta$ is just one move from the goal state; so $r(\beta)=1$, not $m+1$. Thus, $\alpha$ has the maximal finite remoteness of any string in $\textbf{S}$.

This also means that the positions in $\textbf{S}$ that are winning for the prisoner form a cycle. The only question remaining is why this is the same cycle that we would get from the greedy algorithm.

\begin{theorem}
  The greedy algorithm generates the game ``tree'' for the warden's game.
\end{theorem}

\begin{proof}
  Let $\{\beta_m\}\subseteq\textbf{S}$ be the sequence of strings generated by the greedy algorithm, starting from $\alpha$. We have $\beta_0=\alpha$, and for each $m\ge 0$, if $\beta_m=a\gamma$, then $\beta_{m+1}=\gamma b$, where $b$ is the least element of $\{1,\cdots,k\}$ such that $\gamma b\in\textbf{S}$ and $\gamma b$ does not appear in the set $\{\beta_i\}_{i=1}^m$. (If $\gamma b\in \{\beta_i\}_{i=1}^m$ for all $b\in\{1,\cdots,k\}$, then there is no $\beta_{m+1}$; $\beta_m$ is the last string in the sequence.) Obviously, $\beta_0=\alpha$ is the one position of remoteness 0. We must show that $r(\beta_m)=m$ for all $m>0$; we will prove this by induction.
  
  Assume we have $r(\beta_i)=i$ whenever $0<i\le m$. Let $\beta_m=a\gamma$. Assume there is a position of remoteness $m+1$; it must be of the form $\gamma b\in\textbf{S}$ (so that there is a move available to $a\gamma$), where $\gamma b$ is not in $\{\beta_i\}_{i=1}^m$ (since all strings in that set have remoteness $m$ or less). Let $b_1<b_2<\cdots <b_j$ be the elements of $\{1,\cdots k\}$ such that $\gamma b_i\in\textbf{S}$, but $\gamma b_i$ is not in $\{\beta_i\}_{i=1}^m$. (Thus, $\beta_{m+1}=\gamma b_1$.) By Lemma \ref{largerhelpsdwarden}, $r(\gamma b_1)\le r(\gamma b_2)\le\cdots \le r(\gamma b_j)$; by Lemma \ref{onechain}, all of those inequalities are strict except for where we have multiple positions of infinite remoteness. We have $r(\gamma b_1)>m$, so the only $i$ where we can have $r(\gamma b_i)=m+1$ is $i=1$. Thus, we must have $r(\beta_{m+1})=r(\gamma b_1)=m+1$.
  
  Now assume that $m$ is the largest finite remoteness of any position in $\textbf{S}$; that is, $r(\alpha)=m$. The above inductive argument shows that $\beta_m=\alpha$. And the way we defined $\textit{Greedy}_\alpha(\textbf{S})$, the process halts if we ever have $\beta_m=\alpha$. So we do have $r(\beta_m)=m$ for all $m>0$; the length-$n$ substrings of $\textit{Greedy}_\alpha(\textbf{S})$ are exactly the winning positions for the prisoner, in order of remoteness.
\end{proof}

We thus have proven Theorem \ref{mainresult}; the length-$n$ substrings of $\textit{Greedy}_\alpha(\textbf{S})$ are exactly the winning positions for the prisoner, which are exactly the strings in $\textbf{S}$ which are increasable in $\textbf{S}$ to a rotation of $\alpha$.

As a final note for this section, here's a comment on the optimal strategy for the warden:

\begin{corollary}
  Given any position $\gamma b\in\textbf{S}$, if $a$ is the greatest number less than $b$ such that $a\gamma\in\textbf{S}$, then the optimal move for the warden from $\gamma b$ is either to move to $a\gamma$, or to pass. (So, when the warden does decrease a number, he should always do so by the smallest amount possible.)
\end{corollary}

\begin{proof}
  Assume not. Assume there is a position $\gamma b\in\textbf{S}$, where $a$ is the greatest number less than $b$ where $a\gamma\in\textbf{S}$, but the warden's optimal move is to $c\gamma$, where $c<a$. Let $r(\gamma b)=m$; then the remoteness of $c\gamma$ is $m-1$ (where, if $c\gamma=\alpha$, we are treating $\alpha$ as an end position).
  
  Since $\textbf{S}$ is closed under rotations and $a\gamma\in\textbf{S}$, we have $\gamma a\in\textbf{S}$. From Lemmas \ref{largerhelpsdwarden} and \ref{onechain}, since $a<b$, we have $r(\gamma a)<r(\gamma b)$. So $r(\gamma a)<m$. But from the position $\gamma a$, the warden can move to $c\gamma$, a position of remoteness $m-1$. So $r(\gamma a)\ge m$, contradiction.  
\end{proof}

There seems to be no similar statements we can make about the optimal strategy for the prisoner; depending on the situation, the prisoner may want to increase the value on a die by the least amount possible, the greatest amount possible, or some amount in between. For example, all such possiblities occur in our ``no primes'' example, $\textbf{S}_2\subseteq\textbf{T}(2,9)$. There seems to be no way for the prisoner to determine anything about his optimal next move from $\beta$, other than to generate the entire string $\textit{Greedy}_\alpha(\textbf{S})$ until $\beta$ is reached.

\section{Interesting examples}

In \cite{SWW}, there are a number of examples of interesting sets $\textbf{S}\subseteq\textbf{T}(n,k)$ where the greedy algorithm produces a universal cycle for $\textbf{S}$. Here are some new such sets derived from Theorem \ref{mainresult}.

\subsection{Strings increasable to a rotation of $\alpha$}

Choose any $\alpha\in\textbf{T}(n,k)$, and let $\textbf{S}\subseteq\textbf{T}(n,k)$ be the strings that are increasable in $\textbf{T}(n,k)$ to a rotation of $\alpha$. Obviously, $\textbf{S}$ is then closed under rotations.

Given any $\beta\in\textbf{S}$, $\beta$ is increasable in $\textbf{T}(n,k)$ to a rotation of $\alpha$. So there are strings $\beta_0,\beta_1,\ldots,\beta_m\in\textbf{T}(n,k)$ such that $\beta_0=\beta$, $\beta_m$ is a rotation of $\alpha$, and each $\beta_i$ can be changed to $\beta_{i+1}$ by increasing one symbol. Then each $\beta_i$ is increasable in $\textbf{T}(n,k)$ to a rotation of $\alpha$, so each $\beta_i\in\textbf{S}$. But that means $\beta$ is actually increasable in $\textbf{S}$ to a rotation of $\alpha$. Since this is true for all $\beta\in\textbf{S}$, the greedy algorithm starting from $\alpha$ generates a universal cycle of $\textbf{S}$.

For example, let $n=3$, $k=4$, and $\alpha=143$. Here's a universal cycle for the strings in $\textbf{T}(3,4)$ increasable to a rotation of 143:

$$(143)1112113122123132133141142143$$

Note: we get the same collection of strings if $\alpha$ is either 314 or 431. But the resulting universal cycle would be different in either such case. Here's the universal cycle for $\alpha=314$:

$$(314)1112113114212213214312313314$$

And here's the universal cycle for $\alpha=431$:

$$(431)1121131221231321331411421431$$

\subsection{Unions}

Let $\textbf{S}_1,\textbf{S}_2\in\mathcal{G}(n,k)$, where $\textbf{S}_1$ and $\textbf{S}_2$ are both closed under rotations. Assume all strings in $\textbf{S}_1$ and $\textbf{S}_2$ are increasable (in their respective sets) to rotations of a single string $\alpha$. Then all strings in $\textbf{S}_1\cup\textbf{S}_2$ are increasable in $\textbf{S}_1\cup\textbf{S}_2$ to a rotation of $\alpha$, so the greedy algorithm starting from $\alpha$ generates a universal cycle of $\textbf{S}_1\cup\textbf{S}_2$.

This raises the question of whether the same can be said of intersections. However, this turns out to be false: even if the greedy algorithm generates universal cycles for $\textbf{S}_1$ and $\textbf{S}_2$, the same may not be true of $\textbf{S}_1\cap\textbf{S}_2$. One simple example will demonstrate why. Let $\textbf{S}_1,\textbf{S}_2\subseteq\textbf{T}(2,3)$ be as follows:

$$\textbf{S}_1=\{11,13,31,33\}$$

$$\textbf{S}_2=\{11,12,21,23,32,33\}$$

The greedy algorithm (starting from 33) generates universal cycles for both $\textbf{S}_1$ and $\textbf{S}_2$, but does not do so for $\textbf{S}_1\cap\textbf{S}_2=\{11,33\}$. The problem is that there may be an element $\beta\in\textbf{S}_1\cap\textbf{S}_2$ which is increasable to a rotation of $\alpha$ in both $\textbf{S}_1$ and in $\textbf{S}_2$, but the paths from $\beta$ to $\alpha$ may be different in each set. (Here, $\beta=11$; we have $11\rightarrow 13\rightarrow 33$ in $\textbf{S}_1$, and $11\rightarrow 12\rightarrow 32\rightarrow 33$ in $\textbf{S}_2$.)

\subsection{Rotations of increasing strings}

Assume that $n\le k$. Let $\textbf{S}$ be the set containing all strictly increasing strings in $\textbf{T}(n,k)$ and their rotations. For example, in $\textbf{T}(3,5)$, $\textbf{S}$ contains the following strings and their rotations:

$$\{123,124,125,134,135,145,234,235,245,345\}$$

By definition, $\textbf{S}$ is closed under rotations. Let $\alpha$ be the lexicographically maximal, strictly increasing string in $\textbf{T}(n,k)$: $$\alpha=(k-n+1)(k-n+2)\cdots(k-1)k.$$ Any string $\beta\in\textbf{S}$ is increasable in $\textbf{S}$ to a rotation of $\alpha$; the greatest symbol in $\beta$ can be increased to $k$, then the next-greatest symbol can be increased to $k-1$, and so on. So a universal cycle for $\textbf{S}$ can be generated with the greedy algorithm.

\subsection{Maximum cyclic increment or cyclic decrement}

Choose integers $I>0$ and $D>0$. Let $\textbf{S}\subseteq\textbf{T}(n,k)$ be the set of strings with no cyclic increment of size greater than $I$ and no cyclic decrement of size greater than $D$. For example, if we take $\textbf{T}(3,4)$, $I=2$, and $D=1$, then $\textbf{S}$ consists of the following strings and their rotations:

$$\{111,112,122,132,222,223,233,243,333,334,344,444\}$$

By definition, $\textbf{S}$ is closed under rotations. Let $\alpha=k^n$; $\alpha$ contains no cyclic increments or decrements, so $\alpha\in\textbf{S}$.

To show that the greedy algorithm works here: choose any $\beta\in\textbf{S}$  such that $\beta\ne k^n$. Assume $\beta=\gamma_1 a\gamma_2$, where $a$ is the least symbol in $\beta$. Let $\beta'=\gamma_1(a+1)\gamma_2$. This change from $\beta$ to $\beta'$ will either decrease the size of cyclic increments/decrements, or will produce a new cyclic increment or decrement of size 1 (which is legal). So $\beta'\in\textbf{S}$. Thus, for any $\beta\in\textbf{S}$, it's possible to increase a symbol of $\beta$ by 1 to produce another string in $\textbf{S}$. This process can be continued until $\alpha$ is reached. So any $\beta\in\textbf{S}$ is increasable in $\textbf{S}$ to $\alpha$, and the greedy algorithm (starting from $\alpha$) produces a universal cycle.

\subsection{Minimum span, maximum span}

Choose integers $m$ and $M$ such that $0\le m<M<k$. Let $\textbf{S}\subseteq\textbf{T}(n,k)$ be the set of strings $\beta$ whose span is at least $m$ and at most $M$. (The ``span'' of a string $\beta$ is the difference between the least and greatest symbols in $\beta$.) For example, if we take $\textbf{T}(3,4)$, $m=1$, and $M=2$, then $\textbf{S}$ consists of the following strings and their permutations:

$$\{112,113,122,123,133,223,224,233,234,244,334,344\}$$

Clearly, $\textbf{S}$ is closed under rotations. Any $\beta\in\textbf{S}$ can be increased in $\textbf{S}$ to a rotation of $\alpha=(k-m)k^{n-1}$, as follows: if the span of $\beta$ is greater than $m$, increase the least symbol of $\beta$ by 1. If the span of $\beta$ equals $m$, increase the greatest symbol of $\beta$ by 1, unless the greatest symbol is $k$. Repeat this process until a string $\beta$ containing the symbol $k$ is reached. At that point, the least symbol in $\beta$ will be $k-m$; leave that one symbol alone, and increase all the other symbols of $\beta$ to $k$.

Thus, the greedy algorithm generates a universal cycle for $\textbf{S}$.

\subsection{Avoiding a substring}

Choose a string $\gamma\in\textbf{T}(m,k)$ for some $m\ge 1$, and let $\textbf{S}\subseteq\textbf{T}(n,k)$ be the set of strings that do not contain $\gamma$ as a cyclic substring. It was proven in \cite{SWW} that if $\gamma$ does not contain $k$, then $\textbf{S}$ can be generated by the greedy algorithm. If $\gamma$ does contain $k$, then we can still make a weaker statement: let $i,j\in\{1,\ldots,k\}$ be two symbols such that $i<j$. If $\gamma$ contains $i$ but not $j$, then $\textbf{S}$ can be generated by the greedy algorithm.

As an example, let $\textbf{S}\subseteq\textbf{T}(3,3)$ be the set of strings not containing 13 as a cyclic substring. (In this case, $i=1$ and $j=2$.) Then $\textbf{S}$ consists of the rotations of the following strings:

$$\{111,112,122,123,222,223,233,333\}$$

The reason why the greedy algorithm works: $k^n\in\textbf{S}$, since the forbidden substring $\gamma$ includes a symbol $a<k$. Given any $\beta\in\textbf{S}$, we can increase $\beta$ in $\textbf{S}$ to $k^n$, as follows: replace any occurrences of $a$ in $\beta$ with $b$, then increase all symbols in $\beta$ to $k$. So all strings in $\textbf{S}$ are increasable in $\textbf{S}$ to $k^n$.

It would seem to be a difficult question to completely categorize the forbidden substrings $\gamma$ for which $\textbf{S}$ can be generated by the greedy algorithm. I have not found any examples of a forbidden substring $\gamma$ containing a symbol $a\le k-2$ where the greedy algorithm fails. And I would conjecture that there are none:

\begin{conjecture}
  Let $\gamma\in\textbf{T}(m,k)$ be a string containing a symbol $a\le k-2$. Let $\textbf{S}\subseteq\textbf{T}(n,k)$ (for some $n\ge m$) be the set of strings not containing $\gamma$ as a cyclic substring. Then the greedy algorithm starting from $k^n$ generates a universal cycle for $\textbf{S}$.
\end{conjecture}

Now, if all symbols in $\gamma$ are either $k$ or $k-1$, then the greedy algorithm may fail. Let $\textbf{S}\subseteq\textbf{T}(n,9)$ (for some $n\ge 4$) be the set of strings not containing a particular $\gamma$ as a cyclic substring. I leave it as an exercise to the interested reader to show that the greedy algorithm (starting from $\alpha=9^n$) succeeds if $\gamma$ is in the following set...

$$\{8899,8989\}$$

... but the greedy algorithm fails if $\gamma$ is in the following set.

$$\{89,889,899,8889,8999\}$$

Note: for some strings $\gamma$, the outcome depends on the value of $n$. One example is $\gamma=8998$; the greedy algorithm will fail if and only if $n$ is a multiple of 3. (All strings in $\textbf{S}$ will be increasable in $\textbf{S}$ to $9^n$, except for $\beta=(889)^{n/3}$ and its rotations.)

\section{Conclusion and future work: necklaces}

At this point, we have classified the subsets $\textbf{S}\subseteq\textbf{T}(n,k)$, closed under rotations, where the greedy algorithm produces a universal cycle for $\textbf{S}$. But there's one major problem with generating universal cycles with the greedy algorithm: we must store the entire cycle (which could be exponential in length) in order to generate the cycle. Fortunately, when $\textbf{S}=\textbf{T}(n,k)$, there is a faster method:

Given $\alpha\in\textbf{T}(n,k)$, $\alpha$ is a ``necklace'' if, out of all rotations of $\alpha$, $\alpha$ itself is the lexicographically earliest such rotation. If we take all such necklaces in lexicographic order, and append their aperiodic prefixes, we obtain a de Bruijn cycle for $\textbf{T}(n,k)$ (the same cycle produced by the greedy algorithm). For example, here's the resulting universal cycle for $\textbf{T}(3,3)$ (with spaces added between the prefixes):
$$(333)1 \ 112 \ 113 \ 122 \ 123 \ 132 \ 133 \ 2 \ 223 \ 233 \ 3$$
In \cite{SWW}, this is called the FKM algorithm. It is proved in \cite{SWW} that this algorithm produces a universal cycle for $\textbf{S}\subseteq\textbf{T}(n,k)$ (the same universal cycle produced by the greedy algorithm) if
\begin{enumerate}
  \item
  $\textbf{S}$ is closed under rotations, and
  \item
  every necklace in $\textbf{S}$ remains a necklace in $\textbf{S}$ if the suffix of length $i$ (whenever $1\le i\le n$) is replaced with $k^i$. (When this second condition holds true, $\textbf{S}$ is referred to as a ``$k$-suffix language''.)
\end{enumerate}
However, this is not an ``if and only if'' situation. The following example of a set $\textbf{S}\subseteq\textbf{T}(4,3)$ is given in \cite{SWW}:
$$\textbf{S}=\{1112,1121,1122,1211,1212,1221,1222,1322,2111,$$
$$2112,2121,2122,2132,2211,2212,2213,2221,3221\}$$
Not all necklaces in $\textbf{S}$ remain in $\textbf{S}$ when a suffix is replaced with all 3's. However, the FKM algorithm works here. The necklaces in $\textbf{S}$, in lexicographic order, are 1112, 1122, 1212, 1222, and 1322. Reduce 1212 to its aperiodic prefix 12, then concatenate all the strings, and you obtain a universal cycle:
$$(1322)1112 \ 1122 \ 12 \ 1222 \ 1322$$
So it is natural to ask whether there is a necessary and sufficient condition on $\textbf{S}\subseteq\textbf{T}(n,k)$ so that the FKM algorithm generates a universal cycle for $\textbf{S}$. I have a possible candidate for just such a condition:

Let's generalize the concept of a $k$-suffix language as follows. Given a string $\alpha\in\textbf{T}(n,k)$, call a set $\textbf{S}\subseteq\textbf{T}(n,k)$ an ``$\alpha$-suffix language'' if, for any $\beta\in\textbf{S}$, each symbol in $\beta$ is at most the corresponding symbol in $\alpha$, and we obtain another element of $\textbf{S}$ if we replace any suffix of $\beta$ with an equal-length suffix of $\alpha$. That is, if $\beta=b_1\cdots b_n$ and $\alpha=a_1\cdots a_n$, then for all $m$ such that $1\le m\le n$, we have $b_m\le a_m$ and $b_1\cdots b_{m-1}a_m\cdots a_n\in\textbf{S}$. With this definition in place, I would conjecture the following:

\begin{conjecture}
  Let $\textbf{S}\subseteq\textbf{T}(n,k)$ be a set that is closed under rotations. Then the FKM algorithm generates a universal cycle for $\textbf{S}$ if and only if the set of necklaces in $\textbf{S}$ is an $\alpha$-suffix language, where $\alpha$ is the lexicographically maximal necklace in $\textbf{S}$.
\end{conjecture}

The reason for this conjecture: assume that $\textbf{S}\subseteq\textbf{T}(n,k)$ is closed under rotations and is an $\alpha$-suffix language, where $\alpha$ is the lexicographically maximal necklace in $\textbf{S}$. Under these circumstances, it appears that the prisoner has a strategy such that, if the current position is a necklace $\beta\in\textbf{S}$, the prisoner can ensure that the next necklace position reached is lexicographically earlier than $\beta$. Thus, in the universal cycle generated by the greedy algorithm, the necklaces appear in lexicographic order. Perhaps then, the FKM algorithm generates the same cycle as the greedy algorithm.

\end{document}